\documentclass[11pt]{amsart}
\usepackage[T1]{fontenc}
\usepackage{lmodern}
\usepackage[margin=1.1in]{geometry}
\usepackage{amsmath,amssymb,amsthm,mathtools}
\usepackage{booktabs,microtype,cite}
\usepackage[hidelinks]{hyperref}

\newtheorem{theorem}{Theorem}[section]
\newtheorem{proposition}[theorem]{Proposition}
\newtheorem{lemma}[theorem]{Lemma}
\newtheorem{corollary}[theorem]{Corollary}
\theoremstyle{definition}
\newtheorem{example}[theorem]{Example}
\newtheorem*{question}{Question}
\theoremstyle{remark}
\newtheorem{remark}[theorem]{Remark}
\numberwithin{equation}{section}

\newcommand{\F}{\mathbb F}
\newcommand{\Z}{\mathbb Z}
\newcommand{\R}{\mathbb R}
\newcommand{\E}{\mathbb E}
\newcommand{\Lin}{\mathcal L}
\newcommand{\ind}{\mathbf 1}
\DeclareMathOperator{\wt}{wt}
\newcommand{\NL}{\mathcal{NL}}
\DeclareMathOperator{\im}{im}

\title{Linearity Bounds for APN Functions}
\thanks{Research and writing were to a large extent performed by AI. The detailed statement on the use of generative AI is given at the end of this article.}
\author[Christof Beierle]{Christof Beierle}
\address{Faculty of Computer Science, Ruhr University Bochum, Germany}
\email{christof.beierle@rub.de}
\author[Lukas Kölsch]{Lukas Kölsch}
\address{Department of Mathematics \& Statistics, University of South Florida, Tampa, USA}
\email{lukas.koelsch.math@gmail.com}
\date{}
\hypersetup{pdftitle={Linearity bounds for APN functions},
  pdfsubject={Linearity, plateaued components, and Sidon subsets of APN graphs}}

\begin{document}

\begin{abstract}
For $n\ge5$, let $F\colon\F_2^n\to\F_2^n$ be almost perfect nonlinear
and write $N=2^n$. It is proven that the linearity $\Lin(F)$ of $F$,
i.e., the largest absolute Walsh coefficient of a nonzero component,
is at most $N-6$. If $n$ is even, it is further proven that
$\Lin(F)\le N-\sqrt N/4-2$.
These bounds improve the general upper bound of $N-4$ in odd dimension and $N-6$ in even dimensions $n \geq 10$.

It is further proven that, for each fixed $k$,
the $k$-th largest absolute Walsh coefficient among nonzero components,
counted with multiplicity, is at most $(1+O_k(2^{-n}))N/\sqrt{k}$ as $n\to\infty$. The second and fourth largest coefficients
are at most $2\lfloor N/3\rfloor$ and $N/2$, respectively.

Finally, a bound on the linearity in terms of the number $q$ of
nonplateaued nonzero components is derived. 
In odd dimension, for $q>0$, we have
$\Lin(F)^2\le N(1+\sqrt{q(N-1)})$, so that $\Lin(F)/N\to1$ implies $q/N\to1$.
In even dimension, for every $1/2\le C<1$, the condition
$q\le(4C-C^2-1)N/4+1$ implies $\Lin(F)\le CN$. In particular,
$q\le3N/16+2$ implies $\Lin(F)\le N/2$.
\end{abstract}

\maketitle
\section{Introduction}

Let $n$ be a positive integer and write $N=2^n$.
A function $F\colon \F_2^n\to \F_2^n$ is
\emph{almost perfect nonlinear}, or \emph{APN}, if for every $h\ne0$
and every $y\in \F_2^n$ the equation
\[
F(x+h)+F(x) =y
\]
has at most two solutions $x \in \F_2^n$. Since the solutions occur in pairs
$\{x,x+h\}$, this is the smallest possible differential uniformity.
The term \emph{almost perfect nonlinear} goes back to Nyberg and
Knudsen~\cite{NybergKnudsen93}, where such functions were studied in the context of using them as building blocks in symmetric cryptographic primitives.
The APN condition has a combinatorial interpretation.
A subset of a binary vector space is \emph{Sidon} if no four distinct
elements have sum zero, or equivalently if the sums of unordered pairs
of distinct elements are all different~\cite{BabaiSos85}. The function graph
\[
\Gamma_F=\{(x,F(x)):x\in \F_2^n\}
\]
is Sidon if and only if $F$ is APN~\cite{CarletPicek23}. Indeed, two distinct pairs of
solutions to one derivative equation give four distinct graph points
with sum zero, and the converse follows by pairing four such points.

For $b\in \F_2^n$, the Boolean function $f_b=b\cdot F$ from $\F_2^n$ to $\F_2$ is a
\emph{component} of $F$, where $\cdot$ denotes the usual dot product over
$\F_2$. A \emph{nonzero component} will always mean a component
indexed by $b\ne0$. For a Boolean function $f\colon \F_2^n\to\F_2$, its
Walsh transform is
\[
W_f(a)=\sum_{x\in \F_2^n}(-1)^{f(x)+a\cdot x}\qquad(a\in \F_2^n).
\]
The \emph{linearity} of $F$ is
\[
\Lin(F)=\max_{\substack{a\in \F_2^n\\b\in \F_2^n\setminus\{0\}}}
              |W_{f_b}(a)|.
\]
It measures the largest intersection of the graph of $F$ with an affine
hyperplane:
\begin{equation}\label{eq:hyperplanes}
\max_{\mathcal H}|\Gamma_F\cap\mathcal H|
   =\frac{N+\Lin(F)}2.
\end{equation}
Here the maximum is over affine hyperplanes of $\F_2^n\times \F_2^n$.
Indeed, for $b\ne0$, the hyperplane defined by $a\cdot x+b\cdot y=c$ contains
$(N+(-1)^cW_{f_b}(a))/2$ graph points. A hyperplane with $b=0$
has $a\ne0$ and contains $N/2$ graph points. These relations between APN functions,
Sidon sets, and linearity are discussed in~\cite{CP26}.

The corresponding \emph{nonlinearity} is the minimum distance of a
nonzero component from an affine Boolean function:
\begin{equation*}
\NL(F)=\min_{\substack{b\ne0\\ \ell\text{ affine}}}
             d_H(f_b,\ell)
       =\frac{N-\Lin(F)}2.
\end{equation*}
Here $d_H$ denotes Hamming distance.
Indeed, $d_H(f_b,a\cdot x+c)=(N-(-1)^cW_{f_b}(a))/2$.
The APN condition and high nonlinearity are desirable in
cryptography because of their roles in differential and linear
cryptanalysis; see~\cite{CV94,CarletBook}. 

Apart from some specially structured classes of APN functions, they are quite poorly understood in general. In particular, the maximum linearity attainable by APN functions remains unknown. For $n\ge3$, no APN function with linearity exceeding $N/2$ is known.
The classical general bound is $\Lin(F)<N$, i.e., an APN function in these
dimensions has no affine nonzero component~\cite[p.~424]{Carlet10}.
The restriction $n\ge3$ is essential since the map
$F(x_1,x_2)=(x_1x_2,0)$ on $\F_2^2$ is APN and has $\Lin(F)=4=N$.
Czerwinski and Pott~\cite[Corollary~3.7]{CP26} proved the general bound
\begin{equation}\label{eq:previous}
\Lin(F)\le N-4\qquad(n\ge3).
\end{equation}

The proof of~\eqref{eq:previous} uses an upper bound for the size
of Sidon sets in odd-dimensional binary vector spaces.
Recently, Thornburgh~\cite[Corollary~4.3]{Thornburgh26} proved
$\Lin(F)\le N-6$ for even $n\ge4$, using an improved bound
for Sidon sets in dimensions congruent to three modulo four.
The following theorem extends the bound $N-6$ to odd dimensions
and gives an additive improvement growing with $n$ in even dimensions.

\begin{theorem}\label{thm:small}
Let $F\colon\F_2^n\to\F_2^n$ be APN with $n\ge5$. Then,
\[
\Lin(F)\le N-6.
\]
If $n$ is even, we have
\[
\Lin(F)\le N-\frac{\sqrt N}{4}-2.
\]
\end{theorem}

In particular, for every constant $A>0$ there exists $n_0$ such that
$\Lin(F)\le N-A$ for every APN function in every even dimension
$n\ge n_0$.

The proof of the bound $N-6$ excludes a nonzero component at distance two from an affine
function.
Its small support determines congruences for the character sums of a
projected part of the function graph. Combining these congruences
with the second and fourth moments, together with an integrality argument, gives
the exclusion. For even $n$, a refinement of Nyberg's integrality argument~\cite{Nyberg91} gives a lower bound on the sum of the fourth moments of the nonzero components of functions mapping to $\F_2^{n/2+1}$.
The argument applies to maps defined on subsets of the input space,
so it can be used for the projected graph obtained by deleting the
support of a component attaining the linearity.
Averaging over $(n/2+1)$-dimensional subspaces of component indices
and comparing with the second and fourth moments of the projected
Sidon set gives an upper bound
on the linearity. If $\sum_xF(x)=0$,  divisibility of the Walsh coefficients by four gives
$\Lin(F)\le N-8$ and, in even dimension,
$\Lin(F)\le N-\sqrt N/4-4$; this includes APN permutations
(Corollary~\ref{cor:parity}). If $n$ is odd, the upper bound on the linearity can be further improved if $2^{n/2}$ is close to a number divisible by a large power of $2$ (Proposition~\ref{prop:odd-dimensions}).

Stronger dimension-specific bounds are already available in some cases.
The Sidon-set bounds in~\cite[Table~2]{CP26}, together with
\eqref{eq:hyperplanes}, give $\Lin(F)\le8,16,52,122$ for
$n=4,5,6,7$, respectively; the classification
in~\cite{BrinkmannLeander08} improves the dimension-five bound to $12$.
The bound $8$ in dimension four is attained by the function
$x\mapsto x^3$ over $\F_{16}$. For bounds on the linearity of APN power functions,
see~\cite{Carlet18}.

The additive improvement in Theorem~\ref{thm:small} leaves open the following question.

\begin{question}
Does there exist an absolute constant $C<1$ such that
$\Lin(F)\le CN$ for every $n\ge4$ and every APN function
$F\colon\F_2^n\to\F_2^n$?
\end{question}

Equivalently, is $\NL(F)\ge\eta N$ for some absolute $\eta>0$?
Carlet~\cite[Corollary~3]{Carlet21} proved that
\[
\Lin(F)\le\bigl(N^4/2-N^3\bigr)^{1/4}<2^{-1/4}N
\]
when the maximum absolute Walsh coefficient is attained at two
distinct pairs $(a,b)$ with $b\ne0$. The second main result strengthens this conditional
bound by controlling the $k$-th largest coefficient without any
assumption on the multiplicity of the maximum. Its main strength lies in the control of the largest few absolute Walsh coefficients.

\begin{theorem}\label{thm:second}
Let $F\colon \F_2^n\to \F_2^n$ be APN with $n \geq 4$. Denote by $\Lin_k(F)$
the $k$-th entry when the $N(N-1)$ numbers
$|W_{f_b}(a)|$, indexed by $(a,b)\in \F_2^n\times(\F_2^n\setminus\{0\})$,
are arranged in nonincreasing order, with repeated values retained. Then,
for every $1\le k\le N(N-1)$, we have
\begin{equation}\label{eq:walsh-rank}
\Lin_k(F)\le\frac N{\sqrt{k}}
\sqrt{1+\frac{2^{k-1}-1}{N}}.
\end{equation}
Moreover,
\[
\Lin_2(F)\le2\left\lfloor\frac N3\right\rfloor<\frac{2N}{3},
\qquad
\Lin_4(F)\le\frac N2.
\]
\end{theorem}

For every fixed $k$, the theorem gives
$\Lin_k(F)\le (1+O_k(2^{-n}))N/\sqrt{k}$ as $n\to\infty$. Moreover, along any sequence of APN functions $F_n$ in dimension $n$ with $n\to\infty$,
$\Lin(F_n)/N\to1$ implies $\Lin_2(F_n)/N\to0$;
see Remark~\ref{rem:walsh-isolation}.
The proof expands sums of squared Walsh coefficients over affine
subspaces of their indices. Character orthogonality shows that an
unordered pair of distinct points on the function graph contributes positively only
when its sum is orthogonal to every index in the subspace. The Sidon
property bounds the number of such pairs. A partial-sum bound follows
because $k$ indices have affine span of dimension at most $k-1$.

Finally, the third main result establishes an upper bound on the linearity in terms of the number of
nonplateaued components.
A Boolean function $f$ is \emph{plateaued}, in the terminology of
Zheng and Zhang~\cite{ZhengZhang99}, if some $\lambda>0$ satisfies
$W_f(a)\in\{0,\lambda,-\lambda\}$ for all $a$.
A vectorial function is plateaued if every  component is
plateaued; the magnitude $\lambda$ may depend on the component.
Affine Boolean functions are included in this definition.
Quadratic Boolean functions, i.e., functions with
$D_uD_vf$ being constant for all $u,v\in \F_2^n$, where $D_hf(x) = f(x+h) + f(x)$, are plateaued. For background and
combinatorial properties, see~\cite{CarletProuff03,Carlet15,CarletBook,KolschPolujan26}. For further restrictions on the Walsh spectra of quadratic APN
functions, see B\'en\'eteau et al.~\cite{BeneteauEtal26}.
Most known constructions of APN functions yield functions all of whose
nonzero components are plateaued~\cite{KolschPolujan26}.

Write
\[
q=\bigl|\{b\in \F_2^n\setminus\{0\}:b\cdot F
                       \text{ is not plateaued}\}\bigr|.
\]
The case $q=0$ is classical: Since the nonzero Walsh magnitude of
a plateaued Boolean function is a power of two, the exclusion of
affine APN components gives $\Lin(F)\le N/2$ when $n\ge3$.
In odd dimension, a plateaued APN function is almost bent, i.e.,
every nonzero component has Walsh coefficients in
$\{0,\pm\sqrt{2N}\}$, and hence $\Lin(F)=\sqrt{2N}$
(see, e.g., \cite{BergerEtal06,Canteaut01}).

\begin{theorem}\label{thm:count}
Let $F\colon\F_2^n\to\F_2^n$ be APN with $n \geq 4$, and let $q$ be its number of nonplateaued
nonzero components.
\begin{enumerate}
\item If $n$ is odd and $q>0$, then
\begin{equation}\label{eq:odd-count}
\Lin(F)^2\le N\bigl(1+\sqrt{q(N-1)}\bigr).
\end{equation}
\item If $n$ is even, then, for every $1/2\le C<1$, we have
\begin{equation*}
q\le\frac{4C-C^2-1}{4}\,N+1
\quad\Longrightarrow\quad
\Lin(F)\le CN.
\end{equation*}
\end{enumerate}
\end{theorem}

In particular, along a sequence $F_n$ of APN functions in odd dimensions $n$ tending to infinity, we have the implication
\[
\Lin(F_n)/N\to1\quad\Longrightarrow\quad q_n/N\to1,
\]
where $q_n$ counts the nonplateaued nonzero components of $F_n$.
In even dimension, the theorem allows a positive proportion of
nonplateaued components while still guaranteeing $\Lin(F)\le CN$
for each fixed $1/2\le C<1$. In particular,
\begin{equation}
\label{eq:proportion-special}
q\le\frac{3N}{16}+2
\quad\Longrightarrow\quad
\Lin(F)\le\frac N2.
\end{equation}
Proposition~\ref{prop:even-count} expresses the even-dimensional bound
in terms of the minimum distance $t=(N-\Lin(F))/2$ of a nonzero
component from an affine function.

The proofs use the classical APN fourth-moment identity of Chabaud and Vaudenay~\cite{CV94}. In odd dimension, combining this identity with the fourth-moment lower bounds $N^3$ for arbitrary Boolean functions and $2N^3$ for plateaued Boolean functions bounds the fourth moment of each nonplateaued component in terms of $q$. Parseval's identity and Cauchy-Schwarz then bound its individual Walsh coefficients.
In even dimension, the Sidon fourth-moment identity is applied to the graph over the set where a component attaining the linearity differs from its closest affine function. The other components are paired by addition of this component. For pairs in which both components are plateaued, their Walsh spectra impose lower bounds on the restricted fourth moments.

Section~\ref{sec:prelim} recalls basic facts and character-moment
identities. Section~\ref{sec:small} proves Theorem~\ref{thm:small},
Section~\ref{sec:second} proves Theorem~\ref{thm:second}, and
Section~\ref{sec:count} proves Theorem~\ref{thm:count}.

\section{Preliminaries}\label{sec:prelim}

Throughout this article, assume $n \geq 4$ unless explicitly stated otherwise. The usual dot product is used to identify each binary vector space
with its character group. For $a\in \F_2^n$, write
$\chi_a(x)=(-1)^{a\cdot x}$. For $u\colon \F_2^n\to\R$, denote by $\widehat{u}$ its discrete Fourier transform, defined as
\[
\widehat u(a)=\sum_{x\in \F_2^n}u(x)\chi_a(x).
\]
Orthogonality of characters, Fourier inversion, and Parseval's identity yield
\begin{equation*}
\sum_{a\in \F_2^n}\chi_a(x)=N\ind_{\{0\}}(x),\qquad
u(x)=\frac1N\sum_{a\in \F_2^n}\widehat u(a)\chi_a(x),\qquad
\sum_a\widehat u(a)^2=N\sum_xu(x)^2.
\end{equation*}
More generally, for $u,v\colon\F_2^n\to\R$,
Parseval's identity gives
\begin{equation*}
\sum_{a\in\F_2^n}\widehat u(a)\widehat v(a)
=N\sum_{x\in\F_2^n}u(x)v(x).
\end{equation*}

For a Boolean function $f \colon \F_2^n \to \F_2$, write $s_f=(-1)^f$, so that
$W_f=\widehat{s_f}$ and $\sum_aW_f(a)^2=N^2$.
For an affine Boolean function $\ell(x)=a_0\cdot x+c$, we have
\[
W_{f+\ell}(a)=(-1)^cW_f(a+a_0).
\]
Thus affine addition preserves plateauedness and the absolute Walsh
spectrum.
Write $D_hf(x)=f(x+h)+f(x)$ for its derivative in direction $h$ and
$\wt(f)=|\{x:f(x)=1\}|$ for its weight. A Boolean function is called
\emph{balanced} if its weight is $N/2$.

\subsection{Bent functions and affine components}

A Boolean function is \emph{bent}~\cite{Rothaus76} if every Walsh coefficient has
absolute value $\sqrt N$. The following classical facts are included
with proofs; see also~\cite{CarletBook} for a comprehensive textbook on the topic.

\begin{lemma}\label{lem:plateaued-magnitude}
A nonaffine plateaued Boolean function on $\F_2^n$ has all its absolute Walsh
coefficients bounded above by $N/2$.
\end{lemma}

\begin{proof}
Let $\lambda$ be its nonzero Walsh magnitude and let $k$ be the
number of nonzero Walsh coefficients. Parseval's identity gives
$k\lambda^2=N^2$. Since $\lambda$ is a positive integer and $N$ is a
power of two, $\lambda$ is a power of two. If $\lambda=N$, then
$k=1$, and Fourier inversion gives $s_f=\pm\chi_a$ for some $a$;
thus $f$ is affine. Otherwise $\lambda\le N/2$.
\end{proof}

Part~(1) of the next lemma is the classical derivative characterization
of bent Boolean functions. The output-dimension
restriction in part~(2) is due to Nyberg~\cite{Nyberg91}. The proof below follows the Fourier-inversion
and integrality argument of her original proof; see also
\cite[Theorem~4.6 and Remark~4.7]{KolschPolujan24} for a generalization.
Part~(3) follows from~(1) and~(2); this is Carlet's argument as recorded in~\cite[p.~424]{Carlet10}.

\begin{lemma}\label{lem:bent}
Let $n$ and $m$ be positive integers.
\begin{enumerate}
\item A Boolean function on $\F_2^n$ is bent if and only if $D_hf$ is
balanced for every $h\ne0$.
\item If $K\colon \F_2^n\to\F_2^m$ has every nonzero
component bent, then $n$ is even and $m\le n/2$.
\item If $n\ge3$, no map $K\colon \F_2^n\to\F_2^{n-1}$ has a Sidon
graph. In particular, an APN map $F\colon \F_2^n\to \F_2^n$ has no affine
nonzero component.
\end{enumerate}
\end{lemma}

\begin{proof}
Expanding a squared Walsh coefficient and writing $h=x+y$ gives
\[
W_f(a)^2
 =\sum_{h\in \F_2^n}\chi_a(h)\sum_{x\in \F_2^n}(-1)^{D_hf(x)}.
\]
The inner sum is $N$ for $h=0$. By Fourier inversion, all the squares
$W_f(a)^2$ equal $N$ precisely when the inner sum is zero for each
$h\ne0$. This proves (1).

For (2), $n$ has to be even since the Walsh coefficients are all integers.
For each $b\ne0$, write
$W_{b\cdot K}(0)=2^{n/2}\varepsilon_b$, with $\varepsilon_b\in\{1,-1\}$.
Character orthogonality gives, for every $y\in\F_2^m$,
\[
2^m|K^{-1}(y)|
 = \sum_{x\in\F_2^n}\sum_{b\in\F_2^m}
       (-1)^{b\cdot(K(x)+y)}
 = N+\sum_{b\ne0}(-1)^{b\cdot y}W_{b\cdot K}(0)
 = N+2^{n/2}\sum_{b\ne0}\varepsilon_b(-1)^{b\cdot y}.
\]
If $m>n/2$, dividing by $2^{n/2}$ makes the left-hand side an even integer.
The right-hand side is then $2^{n/2}$ plus a sum of $2^m-1$ odd integers, and is therefore odd since $n\ge2$. This is a contradiction.

For (3), a Sidon graph makes $D_hK$ at most two-to-one for every
$h\ne0$. Indeed, two different pairs $\{x,x+h\}$ with the same
image would give four distinct graph points summing to zero. Since
the codomain has $N/2$ elements, every $D_hK$ would be exactly
two-to-one onto the codomain. Every nonzero component of $K$ would
then have all its nonzero derivatives balanced and hence be bent
by (1). This contradicts (2), since $n-1>n/2$ for $n\ge3$.
Finally, an affine nonzero component of $F$ can be made zero by an affine
addition. An invertible linear change of output coordinates then puts this component into a coordinate position. Deleting the resulting zero coordinate would give such a map $K$, a contradiction.
\end{proof}

\subsection{Character-sum moments}
Let $G$ be a binary vector space and let $A\subseteq G$. Write
\[
\alpha_\chi=\sum_{z\in A}\chi(z),\qquad\chi\in\widehat G.
\]
The following identities are straightforward consequences of character orthogonality
and elementary counting, see also 
\cite[Chapter~4]{TaoVu06}.

Expanding and applying character orthogonality gives, for $j\ge1$,
\begin{equation}\label{eq:character-moments}
\sum_{\chi\in\widehat G}\alpha_\chi^j
=|G| \cdot \bigl|\{(z_1,\ldots,z_j)\in A^j:z_1+\cdots+z_j=0\}\bigr|.
\end{equation}
In particular, $\sum_\chi\alpha_\chi^2=|G||A|$. More generally, for every $s\in G$, we have
\begin{equation}\label{eq:shifted-second-moment}
\sum_{\chi\in\widehat G}\chi(s)\alpha_\chi^2
=|G|\cdot
\bigl|\{(z_1,z_2)\in A^2:z_1+z_2=s\}\bigr|
=|G|\cdot|A\cap(A+s)|.
\end{equation}
Indeed, expanding $\alpha_\chi^2$ and applying character
orthogonality to $\chi(s+z_1+z_2)$ leaves precisely the ordered
pairs satisfying $z_1+z_2=s$.

If $A$ is Sidon, every ordered zero-sum quadruple consists of two
equal pairs. There are $|A|+6\binom{|A|}{2}=3|A|^2-2|A|$ such quadruples, so
\begin{equation}\label{eq:sidon-fourth}
\sum_\chi\alpha_\chi^4=|G|(3|A|^2-2|A|).
\end{equation} 

For a Boolean function $f$ on $\F_2^n$, let
\[
M_4(f)=\sum_{a\in \F_2^n}W_f(a)^4.
\]
Applying~\eqref{eq:sidon-fourth} to an APN graph gives the classical
Chabaud-Vaudenay identity~\cite{CV94}, recalled in the next lemma.

\begin{lemma}\label{lem:APNidentity}
If $F\colon \F_2^n\to \F_2^n$ is APN, then
\begin{equation}\label{eq:M4APN}
\sum_{b\ne0}M_4(f_b)=2N^3(N-1).
\end{equation}
\end{lemma}

\begin{proof}
For $A=\Gamma_F\subseteq G=\F_2^n\times \F_2^n$, we have $|A|=N$, $|G|=N^2$,
and $\alpha_{\chi_{a,b}}=W_{f_b}(a)$, where $\chi_{a,b}(x,y)=(-1)^{a\cdot x+b\cdot y}$ for $a,b\in \F_2^n$. Equation~\eqref{eq:sidon-fourth}
gives $\sum_{a,b}W_{f_b}(a)^4=N^2(3N^2-2N)$.
The zero component contributes $N^4$; subtracting it proves
\eqref{eq:M4APN}.
\end{proof}

\section{Bounds on the linearity}\label{sec:small}

The bound~\eqref{eq:previous} has its coding-theoretic origin in
Brouwer and Tolhuizen~\cite{BrouwerTolhuizen93}; see the Sidon-set
formulation in~\cite[Proposition~1.2 and Corollary~3.7]{CP26}. See
\cite{CCZ98} for the relation between APN functions and codes.
For the bound $N-6$ presented here, the technique is to delete points from the function graph of an APN function to impose congruences on the character
sums of the projected Sidon set. These congruences are then combined with the second and fourth moments of those character sums. For even $n$, an integrality argument gives a further inequality
relating the second and fourth moments of the character sums
of the projected Sidon set. Averaging this inequality over subspaces of component
indices and using the same moment identities yields a bound
depending on $n$.

Throughout this section, let $U=\F_2^{n-1}$. All sums of
vectors are taken in the indicated binary vector space. 
The equivalence between the APN condition and the Sidon property
of the graph also applies to maps with a codomain of a different
dimension, using the same condition on the number of solutions
of each derivative equation.

Given a nonzero component $f_{b_0}$ of an APN map $F$ and an affine
function $\ell$,
write $g=f_{b_0}+\ell=\ind_E$, so that $|E|=d_H(f_{b_0},\ell)$.
Complete $b_0$ to a basis $b_0,b_1,\ldots,b_{n-1}$ of $\F_2^n$ and set
$H=(f_{b_1},\ldots,f_{b_{n-1}})\colon \F_2^n\to U$.
Then $(g,H)$ remains APN. Thus it suffices to consider APN
maps $F=(\ind_E,H)$. Input translations, invertible linear changes
of output coordinates, and affine additions to the output also
preserve the APN property. For such a map $F=(\ind_E,H)$, the set
\begin{equation}\label{eq:partialgraph}
\{(x,H(x)):x\in \F_2^n\setminus E\}\subseteq \F_2^n\times U
\end{equation}
is Sidon. Indeed, four points violating the Sidon condition would lift to
four points of $\Gamma_F$ with first output coordinate zero.

\subsection{An observation about derivatives}\label{subsec:derivatives}
The following lemma records two elementary consequences of the
classical characterization of APN functions by
$|\im D_hF|=N/2$ for every $h\ne0$.

\begin{lemma}\label{lem:derivative}
Let $F=(g,H)\colon \F_2^n\to\F_2\times U$ be APN and let
$\sigma=\sum_{x\in \F_2^n}H(x)$. For $h\ne0$ the following hold.
\begin{enumerate}
\item If $D_hg=0$, then $\sigma=0$.
\item If $\wt(D_hg)=2$, then $D_h(H+\sigma g)$ is exactly
two-to-one onto $U$.
\end{enumerate}
\end{lemma}

\begin{proof}
Let $\mathcal P_h=\{\{x,x+h\}:x\in\F_2^n\}$.
These $N/2$ unordered pairs partition $\F_2^n$.
The APN property says that the vectors
\[
F(x)+F(x+h)
 =\bigl(g(x)+g(x+h),H(x)+H(x+h)\bigr),
\]
one for each pair $\{x,x+h\}$ in $\mathcal P_h$, are distinct. Moreover,
\[
\sum_{\{x,x+h\}\in\mathcal P_h}\bigl(H(x)+H(x+h)\bigr)
 =\sum_{x\in\F_2^n}H(x)=\sigma.
\]
If $D_hg=0$, the values $H(x)+H(x+h)$ associated with these
pairs are themselves distinct and therefore run through $U$.
Their sum is zero, proving (1).

If $\wt(D_hg)=2$, there is exactly one pair
$\{x_0,x_0+h\}$ on which $g$ changes.
Let $c=H(x_0)+H(x_0+h)$.
The other $N/2-1$ pairs give distinct values $H(x)+H(x+h)$,
so these occupy $U\setminus\{r\}$ for some $r\in U$.
Their sum is $r$, since $\sum_{y\in U}y=0$.
Thus $\sigma=c+r$.
Replacing $H$ by $H+\sigma g$ leaves the values of $H(x) + H(x+h)$ associated with
the other pairs unchanged and replaces $c$ by $c+\sigma=r$.
Every value in $U$ now occurs on exactly one pair, proving (2).
\end{proof}

This observation also excludes a component that differs from an affine
function at exactly one input, as already implied by~\eqref{eq:previous}.
Indeed, assume for the sake of contradiction that such an APN function exists. Without loss of generality, such a map is of the form $(g,H)$ with $g=\ind_E$,
$|E|=1$, and $\wt(D_hg)=2$ for every $h\ne0$. By Lemma~\ref{lem:derivative}(2), the map
$H+\sigma g\colon \F_2^n\to U$
would have every nonzero derivative exactly two-to-one onto $U$.
Its graph would be Sidon, contradicting Lemma~\ref{lem:bent}(3).
Together with the exclusion of affine components, this recovers
\eqref{eq:previous} for $n \geq 4$. The next step is to exclude support size two.

\subsection{Support of size two}

The argument uses the character sums from Section~\ref{sec:prelim} with
$G=\F_2^n\times U$ and $|G|=N^2/2$. For $A\subseteq G$,
write $\alpha_\chi=\sum_{z\in A}\chi(z)$.
The next lemma will be needed to deduce the final contradiction. The proof uses a polynomial that is nonnegative on $4\Z+2$,
together with the prescribed second and fourth moments. For a related use of power moments and nonnegative polynomials in the context of APN functions, see~\cite[Theorem~3]{CCD99}. Below, $\E$ denotes the arithmetic
mean over the index $i$. 

\begin{lemma}\label{lem:arithmetic}
Let $N\ge16$ be an integer divisible by $16$, and let
$(X_i)_{i=1}^p$ and $(Y_i)_{i=1}^p$ be integer $p$-tuples, where $p\ge1$,
such that
\begin{gather}
X_i\in4\Z+2,\qquad Y_i\in4\Z,\label{eq:arith-congruences}\\
\E X_i^2=N-4,\qquad \E Y_i^2=N,\label{eq:arith-second}\\
\E(X_i^4+Y_i^4)=2N^2+8N-48.\label{eq:arith-fourth}
\end{gather}
Then $N$ is a square and $Y_i^2=N$ for every $i$.
\end{lemma}

\begin{proof}
 Choose $a\in4\Z_{\ge0}$ such that $a^2$ is closest to $N$ among
the squares of multiples of four. This square is unique because
$N\in16\Z$ and the difference between two consecutive such squares is in $32\Z + 16$.
Let
\[
P_a(T)=\bigl((T+a)^2-4\bigr)\bigl((T-a)^2-4\bigr)=T^4-2(a^2+4)T^2+(a^2-4)^2.
\]
For $T\in4\Z+2$, both $T+a$ and $T-a$ are congruent to two
modulo four, so each factor is nonnegative. 
Using~\eqref{eq:arith-second} and~\eqref{eq:arith-fourth}, this yields
\[
\E P_a(X_i)+\E\bigl[(Y_i^2-N)^2\bigr]=(N-a^2)^2.
\]
The term $\E P_a(X_i)$ is nonnegative by~\eqref{eq:arith-congruences}.
By the choice of $a$, every $Y_i$ also satisfies
\[
(Y_i^2-N)^2\ge(a^2-N)^2.
\]
Equality must therefore hold in this last inequality for every $i$.
Since the closest square is unique, this implies $Y_i^2=a^2$ for every $i$.
Taking means gives $N=\E Y_i^2=a^2$, proving both conclusions.
\end{proof}

For even $n\ge4$, the following proposition also follows from
Thornburgh's nonlinearity bound~\cite[Corollary~4.3]{Thornburgh26}.
The proof below applies to both even and odd dimensions.

\begin{proposition}\label{prop:weight-two}
No nonzero component of an APN function $F\colon \F_2^n\to \F_2^n$ differs from an affine function at exactly two inputs.
\end{proposition}

\begin{proof}
Suppose otherwise. Without loss of generality, write $F=(g,H)$ with
$g=\ind_{\{0,d\}}$ for some $d\ne0$. By an affine addition to $H$ one
can force $H(0)=H(d)=0$.
Since $D_dg=0$, Lemma~\ref{lem:derivative} gives $\sum_xH(x)=0$.
Set
\[
A=\{(x,H(x)):x\notin\{0,d\}\}\subseteq G,
\qquad |A|=N-2,\qquad s=(d,0).
\]
The set $A$ is Sidon by~\eqref{eq:partialgraph}, and
\begin{equation}\label{eq:two-sum}
\sum_{z\in A}z=s.
\end{equation}
Moreover, $s$ is not the sum
of two distinct elements of $A$. Indeed, if
$(x,H(x))+(y,H(y))=(d,0)$ for distinct $x,y\notin\{0,d\}$,
then $g(x)=g(y)=0$. The corresponding points
$(x,0,H(x))$ and $(y,0,H(y))$ on the graph of $F$ would have sum $(d,0,0)$, which is
also the sum of the deleted graph points $(0,1,0)$ and $(d,1,0)$.
These four points are distinct, contradicting the Sidon property
of $\Gamma_F$.

For any character $\chi \in \widehat{G}$, Equation~\eqref{eq:two-sum}
implies
\[
\prod_{z\in A}\chi(z)=\chi(s).
\]
Since $|A|=N-2\equiv2\pmod4$, counting the negative signs gives 
\begin{equation}\label{eq:two-congruences}
\alpha_\chi\equiv
\begin{cases}
2\pmod4,&\chi(s)=1,\\
0\pmod4,&\chi(s)=-1.
\end{cases}
\end{equation}
There are $|G|/2$ characters of each type. Identity~(\ref{eq:shifted-second-moment}) and the
absence of a pair summing to $s$ give
\[
\sum_\chi\chi(s)\alpha_\chi^2
 =|G| \cdot \bigl|\{(z_1,z_2)\in A^2:z_1+z_2=s\}\bigr|=0.
\]
Combining this with the second moment in~\eqref{eq:character-moments},
this yields
\begin{equation}\label{eq:two-halves}
\sum_{\chi(s)=1}\alpha_\chi^2
 =\sum_{\chi(s)=-1}\alpha_\chi^2=\frac{|G||A|}2.
\end{equation}

Write the characters as
$\chi_{a,c}(x,y)=(-1)^{a\cdot x+c\cdot y}$, where $a\in \F_2^n$
and $c\in U$. Now consider only those with $c=0$. Their corresponding sums $\alpha_{\chi_{a,0}}$ are
\[
\alpha_{\chi_{a,0}}=N\ind_{\{0\}}(a)-1-(-1)^{a\cdot d}
\]

and they occur with the following values and multiplicities:
\begin{center}
\renewcommand{\arraystretch}{1.2}
\begin{tabular}{@{}cccc@{}}
\toprule
Condition on $a$ & Multiplicity & $\chi_{a,0}(s)$ & $\alpha_{\chi_{a,0}}$\\
\midrule
$a=0$ & $1$ & $1$ & $|A|$\\
$a\ne0,\ a\cdot d=0$ & $N/2-1$ & $1$ & $-2$\\
$a\cdot d=1$ & $N/2$ & $-1$ & $0$\\
\bottomrule
\end{tabular}
\end{center}
Among the characters with $c \neq 0$, each of the two
classes $\chi(s)=1$ and $\chi(s)=-1$ contains exactly
\[
p=\frac N2\left(\frac N2-1\right)
 =\frac{|G|-N}{2}=\frac{N(N-2)}4
\]
characters: there are $N/2-1$ nonzero choices of $c\in U$,
and for each such $c$, exactly $N/2$ choices of $a$ give
each sign of $\chi_{a,c}(s)=(-1)^{a\cdot d}$.
List their values of $\alpha_\chi$ as $X_i$ for $\chi(s)=1$ and as $Y_i$
for $\chi(s)=-1$. Equation~\eqref{eq:two-congruences} gives
\[
X_i\in4\Z+2,\qquad Y_i\in4\Z.
\]
By~\eqref{eq:two-halves}, the sum of the squares over all
characters in each of the two classes is $|G||A|/2$.
In the class $\chi(s)=1$, the removed characters with $c=0$
correspond to values $|A|$ once and $-2$ exactly $N/2-1$ times.
Their total squared contribution is therefore
$|A|^2+4(N/2-1)$. 
In the class $\chi(s)=-1$, all removed characters have value
zero.
Consequently,
\[
\E X_i^2
 =\frac{|G||A|/2-|A|^2-4(N/2-1)}p=N-4,\qquad
\E Y_i^2=\frac{|G||A|}{2p}=N.
\]

Similarly, subtracting the fourth powers of the removed values
from~\eqref{eq:sidon-fourth} gives
\[
\E(X_i^4+Y_i^4)
 =\frac{|G|(3|A|^2-2|A|)-|A|^4-16(N/2-1)}p
 =2N^2+8N-48.
\]
Lemma~\ref{lem:arithmetic} shows that $N$ is a square and that
every character $\chi_{a,c}$ with $c \neq 0$ and $\chi_{a,c}(s)=-1$ satisfies
$\alpha_{\chi_{a,c}}^2=N$. This already yields a contradiction when $n$ is odd.

For even $n$, fix $a\in \F_2^n$ with $a\cdot d=1$. The deleted inputs $(0,H(0))$ and $(d,H(d))$
contribute opposite signs, since $H(0)=H(d)=0$, so, for every $c\in U$, we have
\[
\alpha_{\chi_{a,c}}=\sum_{x\in \F_2^n}(-1)^{a\cdot x+c\cdot H(x)}.
\]
The value for $c=0$ is zero. For $c\ne0$, we have
$\chi_{a,c}(s)=(-1)^{a\cdot d}=-1$, so
$\alpha_{\chi_{a,c}}$ is one of the values $Y_i$.
Therefore,
$\alpha_{\chi_{a,c}}\in\{\sqrt N,-\sqrt N\}$. Character orthogonality now yields
\begin{equation}\label{eq:two-parity}
\sum_{c\ne0}\alpha_{\chi_{a,c}}
 =\sum_{c\in U}\sum_{x\in \F_2^n}(-1)^{a\cdot x+c\cdot H(x)}
 =\frac N2\sum_{x:H(x)=0}(-1)^{a\cdot x}.
\end{equation}
The left-hand side is $\sqrt N$ times an odd integer, since $|U|-1=N/2-1$
is odd. The right-hand side is divisible by $N/2$, whereas
\[
\frac{N/2}{\sqrt N}=2^{n/2-1}
\]
is even for even $n\ge4$. Dividing~\eqref{eq:two-parity} by
$\sqrt N$ thus yields an odd integer being equal to an even integer, a
contradiction.
\end{proof}

\subsection{Even dimension}

Nyberg's bound on the output dimension of a vectorial bent
function~\cite{Nyberg91}, recalled in Lemma~\ref{lem:bent}(2),
follows from Fourier inversion and an integrality argument.
For functions defined on a subset of $\F_2^n$ mapping to $n/2+1$ bits, the next lemma uses a similar parity argument to bound from below the sum of the
differences between the fourth moments of the nonzero components and the value $N^3$ attained
by bent functions.

\begin{lemma}\label{lem:even-moments}
Let $n\ge4$ be even,$r=n/2+1$, and let
$K\colon D\to\F_2^r$, where $D\subseteq\F_2^n$. Define
\[
T_b(a)=\sum_{x\in D}(-1)^{b\cdot K(x) + a \cdot x}
\qquad
(a\in\F_2^n,\ b\in\F_2^r).
\]
Then,
\[
\sum_{b\in\F_2^r\setminus\{0\}}\sum_{a\in\F_2^n}
\bigl(T_b(a)^2-N\bigr)^2\ge N^2|D|.
\]
In particular, if $D=\F_2^n$, we have
\[
\sum_{b\in\F_2^r\setminus\{0\}}
\bigl(M_4(b\cdot K)-N^3\bigr)\ge N^3.
\]
\end{lemma}

\begin{proof}
Write $v=\sqrt N$ so that $2^r=2v$.
Fix $a\in\F_2^n$. For each $b\ne0$, choose
$s_b\in\{v,-v\}$ that is closest to $T_b(a)$, and let
$s_0=N\ind_{\{0\}}(a)$.

The Fourier transform of the function $b \mapsto T_b(a)$ is valued in $2v \Z$,
as
\[
\sum_b T_b(a)\chi_b(y)
=
2v\sum_{\substack{x\in D\\K(x)=y}}(-1)^{a\cdot x}.
\]
On the other hand, the Fourier transform of $b \mapsto s_b$ takes values in
$2v\Z + v$. Indeed, for every $y\in\F_2^r$, we have
\[
\sum_b s_b\chi_b(y)
=
N\ind_{\{0\}}(a)
+\sum_{b\ne0}s_b(-1)^{b\cdot y},
\]
which is in $2v \Z + v$ because
$N$ is divisible by $2v$, and the final sum consists of
an odd number of terms in $\{v,-v\}$.

The two Fourier transforms therefore differ by at least $v$
at every point. Parseval's identity applied to $b \mapsto T_b(a)-s_b$ yields
\begin{equation}
\label{eq:integrality-pointwise}
\bigl(T_0(a)-N\ind_{\{0\}}(a)\bigr)^2
+\sum_{b\ne0}\bigl(|T_b(a)|-v\bigr)^2
\ge \frac{1}{2v}\cdot 2v \cdot v^2=N.
\end{equation}
Since
\[
T_0(a)-N\ind_{\{0\}}(a)
=\sum_{x\in\F_2^n}\bigl(\ind_D(x)-1\bigr)\chi_a(x),
\]
Parseval's identity applied to $x\mapsto\ind_D(x)-1$ yields
\[
\sum_a\bigl(T_0(a)-N\ind_{\{0\}}(a)\bigr)^2
=N\sum_x\bigl(\ind_D(x)-1\bigr)^2
=N(N-|D|).
\]
Summing Inequality~\eqref{eq:integrality-pointwise} over $a$ thus gives
\begin{equation}\label{eq:even-walsh-distance}
\sum_{b\ne0}\sum_a\bigl(|T_b(a)|-v\bigr)^2
\ge N|D|.
\end{equation}

The first assertion follows from
$(w^2-N)^2=(|w|+v)^2(|w|-v)^2
\ge N(|w|-v)^2$.
When $D=\F_2^n$, Parseval also gives
\[
\sum_a\bigl(T_b(a)^2-N\bigr)^2
=M_4(b\cdot K)-N^3,
\]
proving the second assertion.
\end{proof}

The key step is now to average the inequality in Lemma~\ref{lem:even-moments} over all
$(n/2+1)$-dimensional subspaces of $U = \F_2^{n-1}$. Since every nonzero
component index occurs equally often among all those subspaces, the individual bounds yield a bound on the sum over all nonzero indices,
which can then be compared with the Sidon moment identities.
\begin{proposition}\label{prop:even-linearity}
Let $F\colon\F_2^n\to\F_2^n$ be APN with $n\ge4$ even. Then
\[
\Lin(F)<N-\frac{\sqrt N}{4}.
\]
\end{proposition}

\begin{proof}
Write $t=(N-\Lin(F))/2$.
Without loss of generality, assume
$F=(\ind_E,H)$, where $|E|=t$ and
$H\colon\F_2^n\to U$.
For $D=\F_2^n\setminus E$, the set $\{(x,H(x)) \colon x \in D\}$ is Sidon. Let
\[
\alpha_{\chi_{a,c}}=\sum_{x\in D}(-1)^{a\cdot x+c\cdot H(x)},
\qquad
R=\sum_{c\ne0}\sum_a(\alpha_{\chi_{a,c}}^2-N)^2.
\]

Apply Lemma~\ref{lem:even-moments} to the components of $H|_D$
indexed by each $(n/2+1)$-dimensional subspace of $U$. Every nonzero index belongs to the same proportion $(2\sqrt N-1)/(N/2-1)$ of these subspaces. Denoting by $\mathcal B$ the family of such subspaces, averaging the inequality in Lemma~\ref{lem:even-moments} therefore gives
\[
N^2 |D|
\le
\frac1{|\mathcal B|}
\sum_{B\in\mathcal B}
\sum_{c\in B\setminus\{0\}}\sum_a(\alpha_{\chi_{a,c}}^2-N)^2
=
\frac{2\sqrt N-1}{N/2-1}R.
\]
By rearranging terms, we have
\begin{equation}\label{eq:even-fourth-bound}
R\ge\frac{N-2}{2(2\sqrt N-1)}N^2|D|.
\end{equation}

To derive an upper bound on $R$, Equations~\eqref{eq:character-moments} and
\eqref{eq:sidon-fourth} give
\[
\sum_a\alpha_{\chi_{a,c}}^2=N|D|,
\qquad
\sum_{a,c}\alpha_{\chi_{a,c}}^4
=\frac{N^2}{2}(3|D|^2-2|D|).
\]
Since $\alpha_{\chi_{0,0}}=|D|$ and all fourth powers are nonnegative, we have
$\sum_a\alpha_{\chi_{a,0}}^4\ge |D|^4$.
From $t\le N/2$, it follows that $2|D|-N\ge0$, and
$|D|^2-N(2|D|-N)=t^2\ge0$ therefore gives
$|D|^4\ge N^2(2|D|-N)^2$. Expanding $R$ and substituting $|D| = N-t$, we finally get
\begin{align*}
R\le \frac{N^2}{2}(3|D|^2-2|D|)-|D|^4
   -N^2(2|D|-N)\left(\frac N2-1\right)\le N^2t\left(2|D|-1-\frac t2\right)
<2N^2t|D|,
\end{align*}
where the last inequality is strict because $t>0$.
Combining this with~\eqref{eq:even-fourth-bound} yields
\[
\Lin(F)=N-2t
<N-\frac{N-2}{2(2\sqrt N-1)}
\le N-\frac{\sqrt N}{4},
\]
where the last inequality uses $\sqrt N\ge4$.
\end{proof}

\subsection{The general bounds}

\begin{proof}[Proof of Theorem~\ref{thm:small}]
Let $t=\NL(F)=(N-\Lin(F))/2$, which is the minimum distance of a nonzero
component from an affine function. Lemma~\ref{lem:bent} excludes $t=0$,
and the argument in Section~\ref{subsec:derivatives} excludes $t=1$.
Proposition~\ref{prop:weight-two} excludes $t=2$, yielding
$\Lin(F)\le N-6$.

Suppose now that $n$ is even. Proposition~\ref{prop:even-linearity}
gives $\Lin(F)<N-\sqrt N/4$.
Since $N-\sqrt N/4$ is an even integer and every Walsh coefficient is even, it follows that
\[
\Lin(F)\le N-\frac{\sqrt N}{4}-2.
\]
\end{proof}

\begin{corollary}\label{cor:parity}
Let $F\colon\F_2^n\to\F_2^n$ be APN with $n \geq 5$. If
$\sum_{x\in\F_2^n}F(x)=0$, then
\[
\Lin(F)\le N-8.
\]
If $n$ is even, we have
\[
\Lin(F)\le N-\frac{\sqrt N}{4}-4.
\]
In particular, these bounds hold for every APN permutation.
\end{corollary}

\begin{proof}
For every $b$, the parity of $\wt(b\cdot F)$ is
$b\cdot\sum_xF(x)=0$. Adding an affine function preserves this
parity, since every affine Boolean function on $\F_2^n$ has even weight.
Thus $t=\NL(F)=(N-\Lin(F))/2$ is even.
Theorem~\ref{thm:small} gives $t\ge3$, hence $t\ge4$ and
$\Lin(F)\le N-8$.

Moreover, $\Lin(F)=N-2t$ is divisible by four.
If $n\ge8$ is even, then $N-\sqrt N/4$ is also divisible by four,
so Proposition~\ref{prop:even-linearity} gives
\[
\Lin(F)\le N-\frac{\sqrt N}{4}-4.
\]
For $n=6$, this follows from the bound $N-8$.
\end{proof}

It is conjectured that $\sum_{x\in\F_2^n}F(x)=0$ for every APN function
on $\F_2^n$ with $n\ge3$; see~\cite{BudaghyanEtal18}.
Under this conjecture, the bounds in Corollary~\ref{cor:parity} apply to every
APN function with $n\ge5$.

\subsection{An improvement in some odd dimensions.}
If $n$ is odd, the upper bound of $N-6$ on the linearity can be improved if $2^{n/2}$ is close to a number divisible by a large power of $2$.
\begin{proposition}\label{prop:odd-dimensions}
Let $F\colon\F_2^n\to\F_2^n$ be APN, with $n\ge5$ odd. Choose
$p\in\{\lfloor\sqrt N\rfloor,\lceil\sqrt N\rceil\}$
and let $2^s$ be the largest power of two dividing $p$.
If $s\ge5$, then
\[
\Lin(F)\le N-2^{s-2}-2.
\]
\end{proposition}
\begin{proof}
We have
\[
|p-\sqrt N|<1,\qquad
p\equiv2^s\pmod{2^{s+1}},\qquad
2^{s+1}\le N/2.
\]
Neither nearest integer to $\sqrt N$ is a power of two, so $p$ has an odd prime divisor, which means that
$p\ge3\cdot2^s\ge96$.

Fix an $(s+1)$-dimensional subspace $B\leq \F_2^n$ and  $a \in \F_2^n$. The proof now essentially repeats the Fourier-transform comparison
used to prove Inequality~\eqref{eq:integrality-pointwise}.
For each $b\in B\setminus\{0\}$, choose
$\sigma_b\in\{p,-p\}$ closest to $W_{b\cdot F}(a)$,
and write $\sigma_0=N\ind_{\{0\}}(a)$.

By orthogonality of characters, the Fourier transform on $B$
of $b\mapsto W_{b\cdot F}(a)$ takes values in
$|B|\Z=2^{s+1}\Z$.
The Fourier transform of $b\mapsto\sigma_b$ takes values in
$2^{s+1}\Z+2^s$.
Indeed, $\sigma_0$ is divisible by $2^{s+1}$, while each
of the $|B|-1$ remaining terms is congruent to $2^s$
modulo $2^{s+1}$, and $|B|-1$ is odd.

The two Fourier transforms therefore differ in absolute
value by at least $2^s$ at every character.
Since $W_{0\cdot F}(a)=\sigma_0$, Parseval's identity gives
\[
\begin{aligned}
&\sum_{b\in B\setminus\{0\}}
  \bigl(|W_{b\cdot F}(a)|-p\bigr)^2=\sum_{b\in B}|W_{b\cdot F}(a)-\sigma_b|^2\\
&\quad=\frac1{|B|}
 \sum_{\chi\in\widehat B}
 \left|
 \sum_{b\in B}
 \bigl(W_{b\cdot F}(a)-\sigma_b\bigr)\chi(b)
 \right|^2\ge\frac1{|B|}\,|B|\,2^{2s}
=2^{2s}.
\end{aligned}
\]

If every $b\in B\setminus\{0\}$ satisfies
$\bigl||W_{b\cdot F}(a)|-p\bigr|<2^s$,
then $|W_{b\cdot F}(a)|>p-2^s$, so
$|W_{b\cdot F}(a)|+p>2p-2^s>0$. Hence
\[
\begin{aligned}
&\sum_{b\in B\setminus\{0\}}
  \bigl(|W_{b\cdot F}(a)|^2-p^2\bigr)^2=
\sum_{b\in B\setminus\{0\}}
  \bigl(|W_{b\cdot F}(a)|-p\bigr)^2
  \bigl(|W_{b\cdot F}(a)|+p\bigr)^2\\
&\quad\ge
(2p-2^s)^2
\sum_{b\in B\setminus\{0\}}
  \bigl(|W_{b\cdot F}(a)|-p\bigr)^2\ge 2^{2s}(2p-2^s)^2.
\end{aligned}
\]
If some $b\in B\setminus\{0\}$ satisfies
$\bigl||W_{b\cdot F}(a)|-p\bigr|\ge2^s$,
that single term suffices. Indeed, either
$|W_{b\cdot F}(a)|\le p-2^s$, giving
\[
p^2-|W_{b\cdot F}(a)|^2
\ge p^2-(p-2^s)^2
=2^s(2p-2^s),
\]
or $|W_{b\cdot F}(a)|\ge p+2^s$, giving
\[
|W_{b\cdot F}(a)|^2-p^2
\ge(p+2^s)^2-p^2
=2^s(2p+2^s)
\ge2^s(2p-2^s).
\]
In either case we get
\begin{equation} \label{eq:Walsh_comp}
    \sum_{b\in B\setminus\{0\}}
  \bigl(|W_{b\cdot F}(a)|^2-p^2\bigr)^2 \geq 2^{2s}(2p-2^s)^2.
\end{equation}

For each $b \neq 0$, write
$E_b=M_4(b\cdot F)-N^3$.
For each fixed $b$, expanding the square and using
Parseval's identity gives
\[
\sum_{a\in\F_2^n}
\bigl(W_{b\cdot F}(a)^2-p^2\bigr)^2
=M_4(b\cdot F)-2p^2N^2+Np^4=E_b+N(p^2-N)^2.
\]
Since $|B\setminus\{0\}|=2^{s+1}-1$, it follows that

\begin{align*}
\sum_{b\in B\setminus\{0\}}E_b
&=\sum_{a\in\F_2^n}\sum_{b\in B\setminus\{0\}}
  \bigl(W_{b\cdot F}(a)^2-p^2\bigr)^2-N(2^{s+1}-1)(p^2-N)^2\\
&\ge N\bigl[
  2^{2s}(2p-2^s)^2
  -(2^{s+1}-1)(p^2-N)^2
  \bigr],
\end{align*}
using Eq.~\eqref{eq:Walsh_comp}.
Since $2^s\le p/3$, $p\ge96$, and $\sqrt N<p+1$, we have
\[
\frac{(2p-2^s)^2}{N}
>
\frac{25}{9}\left(\frac{p}{p+1}\right)^2
\ge
\frac{25}{9}\left(\frac{96}{97}\right)^2
>\frac52.
\]
Moreover, $|p-\sqrt N|<1$ and $N\ge32$ give
\[
(p^2-N)^2
=(p-\sqrt N)^2(p+\sqrt N)^2
<(2\sqrt N+1)^2
<5N.
\]
Using these estimates, we obtain

\begin{align}
\label{eq:eb}
\sum_{b\in B\setminus\{0\}}E_b
&>\frac52N^2\,2^{2s}
  -5N^2(2^{s+1}-1) =\frac52N^2\bigl(2^{2s}-4\cdot2^s+2\bigr) >\frac52N^2\,2^s(2^s-4). 
\end{align}

Choose $b_0\ne0$ and $a_0\in\F_2^n$ such that
$|W_{b_0\cdot F}(a_0)|=\Lin(F)$.
Lemma~\ref{lem:APNidentity} gives $\sum_{b\neq 0}E_b=N^4-N^3$, so
\begin{align}
\label{eq:fourth_moment_odd}
\Lin(F)^4
\le M_4(b_0\cdot F) 
=E_{b_0}+N^3=N^4-\sum_{b\notin\{0,b_0\}}E_b. 
\end{align}

Let $\mathcal B$ be the family of all $(s+1)$-dimensional
subspaces of $\F_2^n$ that do not contain $b_0$.
By symmetry, each index $b\notin\{0,b_0\}$ belongs to the
same number of members of $\mathcal B$.
Since each member contains $2^{s+1}-1$ nonzero indices,
each such index belongs to the fraction
$(2^{s+1}-1)/(N-2)$ of these subspaces.
Averaging the preceding lower bound over $\mathcal B$
therefore yields
\[
\frac{2^{s+1}-1}{N-2}
\sum_{b\notin\{0,b_0\}}E_b
=\frac1{|\mathcal B|}
  \sum_{B\in\mathcal B}
  \sum_{b\in B\setminus\{0\}}E_b>\frac52N^2\,2^s(2^s-4),
\]
using Eq.~\eqref{eq:eb}.
Combining this with Eq.~\eqref{eq:fourth_moment_odd} yields
\[
\Lin(F)^4
<N^4-
\frac{5N^2(N-2)\,2^s(2^s-4)}
     {2(2^{s+1}-1)}.
\]
Finally, since $N\ge32$ and $2^s\ge32$, we have
\[
\frac{
  5N^2(N-2)\,2^s(2^s-4)
}{
  2(2^{s+1}-1)\,2^sN^3
}
=
\frac52\left(1-\frac2N\right)
\frac{2^s-4}{2^{s+1}-1}
\ge
\frac52\cdot\frac{15}{16}\cdot\frac{28}{63}
=\frac{25}{24}>1.
\]
Thus
\[
\Lin(F)^4<N^4-2^sN^3.
\]

Taking fourth roots and using Bernoulli’s inequality
$(1 - z)^{1/4} < 1 - z/4$ for nonzero $z < 1$ gives
\[
\Lin(F)
<N\left(1-\frac{2^s}{N}\right)^{1/4}
<N-2^{s-2}.
\]
Both $\Lin(F)$ and $N-2^{s-2}$ are even integers.
Therefore $\Lin(F)\le N-2^{s-2}-2$.
\end{proof}

For example, $5792<\sqrt{2^{25}}<5793$, and $5792=181\cdot2^5$.
Applying Proposition~\ref{prop:odd-dimensions} with $s=5$
gives
\[
\Lin(F)\le2^{25}-10
\]
for every APN function $F\colon\F_2^{25}\to\F_2^{25}$.

\section{Large Walsh coefficients}\label{sec:second}

Czerwinski and Pott~\cite[Sections~2 and~3]{CP26} relate Walsh
coefficients to intersections of APN graphs with affine subspaces.
This viewpoint is useful to bound sums of squared Walsh coefficients over
affine subspaces of their indices. Throughout this section,
$F\colon\F_2^n\to\F_2^n$ is APN.
Write $G=\F_2^n\times\F_2^n$ and $W(a,b)=W_{f_b}(a)$, including $b=0$,
so that $W(0)=N$.

\begin{lemma}\label{lem:affine-walsh}
Let $U\subseteq G$ be an affine subspace of dimension $r$. Then
\[
\sum_{u\in U\setminus\{0\}}W(u)^2\le
\begin{cases}
N^2+(2^r-2)N,&0\in U,\\
N^2+(2^r-1)N,&0\notin U.
\end{cases}
\]
\end{lemma}

\begin{proof}
Let $V=\{0\}\times\F_2^n$ 
 and $d=\dim\langle U\rangle$,
where $\langle U\rangle$ denotes the linear span of $U$.
For every $z\in G$, orthogonality of characters gives
\begin{equation}
\label{eq:character-U}
\sum_{u\in U}(-1)^{u\cdot z}\le2^r\ind_{\langle U\rangle^\perp}(z).
\end{equation}
Indeed, since $U$ is affine, this sum is $0$ or $\pm2^r$,
and it equals $2^r$ precisely when $z$ is orthogonal to every
element of $U$.

Expanding the squared Walsh coefficients gives
\begin{align*}
\sum_{u\in U}W(u)^2
&=\sum_{u\in U}\sum_{z,z'\in\Gamma_F}
    (-1)^{u\cdot(z+z')}\\
&=2^rN+2\sum_{\{z,z'\}\in\binom{\Gamma_F}{2}}
    \sum_{u\in U}(-1)^{u\cdot(z+z')}\\
&\le2^rN+2^{r+1}|\langle U\rangle^\perp \setminus V|.
\end{align*}
For the inequality, use~(\ref{eq:character-U}) and note
that the sums $z+z'$ are all distinct by the Sidon property of $\Gamma_F$ and
have nonzero first coordinate.
The dimension formula gives
\[
|\langle U\rangle^\perp|=\frac{N^2}{2^d},\qquad
|\langle U\rangle^\perp\cap V|\ge\frac N{2^d}.
\]
Consequently,
\[
\sum_{u\in U}W(u)^2
\le2^rN+2^{r+1-d}(N^2-N).
\]
If $0\in U$, then $d=r$, and subtracting $W(0)^2=N^2$ gives
the first bound. If $0\notin U$, then $d=r+1$, yielding the second bound.
\end{proof}

\begin{proof}[Proof of Theorem~\ref{thm:second}]
Choose $k$ distinct indices attaining $\Lin_1(F),\ldots,\Lin_k(F)$,
and let $U$ be their affine span. Its dimension $r$ is at most $k-1$.
All selected indices are nonzero, so Lemma~\ref{lem:affine-walsh} gives
\begin{equation}
\label{eq:walsh-partial-sums}
\sum_{i=1}^k\Lin_i(F)^2
\le\sum_{u\in U\setminus\{0\}}W(u)^2
\le N^2+(2^r-1)N
\le N^2+(2^{k-1}-1)N.
\end{equation}

Since $k\Lin_k(F)^2\le\sum_{i=1}^k\Lin_i(F)^2$, the claimed upper bound on $\Lin_k(F)$ follows.

To bound $\Lin_2(F)$, choose two distinct pairs $(a_i,b_i)$ with
$b_i\ne0$, and write $u_i=|W_{f_{b_i}}(a_i)|$ for $i=1,2$.
Let $m=\min\{u_1,u_2\}$, and suppose that $m>2N/3$.
The two indices are linearly independent. Applying the first case of
Lemma~\ref{lem:affine-walsh} to their span and writing
$w=W_{f_{b_1+b_2}}(a_1+a_2)$ gives
\begin{equation}
\label{eq:ineq-1}
u_1^2+u_2^2+w^2\le N^2+2N.
\end{equation}

Choose signs $\epsilon_i\in\{1,-1\}$ so that
$s_i(x)=\epsilon_i(-1)^{a_i\cdot x+f_{b_i}(x)}$ satisfies
$\sum_xs_i(x)=u_i$. Since $s_1s_2\ge s_1+s_2-1$ holds pointwise and
$\sum_xs_1(x)s_2(x)=\epsilon_1\epsilon_2w$, we have
\begin{equation}
\label{eq:ineq-2}
|w|\ge u_1+u_2-N\ge2m-N>0.
\end{equation}
Combining~(\ref{eq:ineq-1}) and~(\ref{eq:ineq-2}) yields
\[
2m^2+(2m-N)^2\le N^2+2N,
\qquad\text{hence}\qquad m(3m-2N)\le N.
\]
Every Walsh coefficient is even, so $m>2N/3$ implies that
$3m-2N$ is a positive even integer. The left-hand side is then at least
$2m$, a contradiction since $2m >N$. Thus $m\le2N/3$, and $m$ even gives
$m\le2\lfloor N/3\rfloor$. Taking $u_1$ and $u_2$ as the two largest absolute Walsh coefficients
proves the bound on $\Lin_2(F)$; the strict inequality follows from $3\nmid N$.

Finally,~\eqref{eq:walsh-partial-sums} gives
$4\Lin_4(F)^2\le N^2+7N$.
Both $\Lin_4(F)$ and $N/2$ are even, so $\Lin_4(F)>N/2$ would imply
\[
4\Lin_4(F)^2\ge4(N/2+2)^2=N^2+8N+16,
\]
a contradiction.
\end{proof}

\begin{remark}\label{rem:walsh-isolation}
An absolute Walsh coefficient close to $N$ forces all the others
to be small. More precisely, if $\Lin(F) = \Lin_1(F) =(1-\varepsilon)N$ for $0<\varepsilon<1$, then
\begin{equation}\label{eq:walsh-isolation}
\frac{\Lin_2(F)}N
\le\frac{\varepsilon}{2}
   +\sqrt{\varepsilon-\frac{3\varepsilon^2}{4}+\frac1N}.
\end{equation}
In particular, along any sequence $F_n$ of APN functions in dimension $n$ with
$n\to\infty$, we have
\[
\frac{\Lin(F_n)}N\to1
\quad\Longrightarrow\quad
\frac{\Lin_2(F_n)}N\to0.
\]
Indeed, the arguments in the preceding proof give
\[
\Lin_1(F)^2+\Lin_2(F)^2+\max\{\Lin_1(F)+\Lin_2(F)-N,0\}^2\le N^2+2N.
\]
If $\Lin_2(F)>\varepsilon N$, substituting $\Lin_1(F)=(1-\varepsilon)N$ and using $\Lin_1(F) + \Lin_2(F) - N = \Lin_2(F)- \varepsilon N>0$ yields
\[
\left(\frac {\Lin_2(F)}{N}-\frac{\varepsilon}{2}\right)^2
\le\varepsilon-\frac{3\varepsilon^2}{4}+\frac1N,
\]
which proves~\eqref{eq:walsh-isolation}.
If $\Lin_2(F)\le\varepsilon N$, the same bound is immediate, since the
right-hand side of~\eqref{eq:walsh-isolation} is at least $\varepsilon$.
\end{remark}

\begin{example}
\label{ex:walsh-four-sharp}
The bound $\Lin_4(F)\le N/2$ in Theorem~\ref{thm:second}
is attained by some quadratic APN functions in small dimension. Indeed, suppose that $F$ is
quadratic and $\Lin(F)=N/2$. A component attaining this value is
plateaued, since every quadratic Boolean function is plateaued
\cite{CarletBook}. Its nonzero Walsh coefficients therefore all have
absolute value $N/2$, and Parseval's identity shows that there are
exactly
\[
\frac{N^2}{(N/2)^2}=4
\]
such coefficients. Consequently,
\[
\Lin_1(F)=\Lin_2(F)=\Lin_3(F)=\Lin_4(F)=N/2.
\]
For examples see Beierle and Leander~\cite{BeierleLeander22}, who
found quadratic APN functions on $\F_2^8$  with linearity $N/2$ by a computer search. No such APN function is currently known for $n>8$.
\end{example}

\section{Nonplateaued components}\label{sec:count}

The two parts of Theorem~\ref{thm:count} use fourth moments of character
sums in different ways. In odd dimension, the minimum fourth moment
of the Walsh coefficients of a plateaued component is larger than the
Parseval lower bound for the Walsh coefficients of a general Boolean
function. In even dimension, the trick is to restrict the graph to the
set where a component attaining the linearity differs from a closest
affine function, and to compare pairs of plateaued components.

Throughout this section, $F\colon \F_2^n\to \F_2^n$ is APN.
For a Boolean function $f$ write
\[
\kappa(f)=\frac{M_4(f)}{N^3}.
\]
By Lemma~\ref{lem:APNidentity}, we have
\begin{equation}\label{eq:kappa-total}
\sum_{b\ne0}\kappa(f_b)=2(N-1).
\end{equation}

The same normalization and identity are used by Gillot, Langevin and
Lo~\cite{GillotLangevinLo25} and by Gillot and
Langevin~\cite{GillotLangevin26} to constrain
the numbers of components with prescribed fourth Walsh moments,
particularly in their study of six-dimensional APN functions with
two distinct such moments. Carlet and Thornburgh~\cite[Remark~2.3]{CarletThornburgh26} also combine
the Sidon fourth-moment identity with Parseval's identity for
restrictions to affine subspaces on which the function is plateaued. 

The bounds below use the distinction between plateaued and
nonplateaued components, together with the Parseval sum outside a
chosen Walsh coefficient in odd dimension and fourth moments of
character sums over a restricted graph in even dimension.

\begin{lemma}\label{lem:kappa}
Every Boolean function satisfies $\kappa(f)\ge1$, with equality
if and only if it is bent. If $f$ is plateaued with
magnitude $\lambda$, then $\kappa(f)=\lambda^2/N$; in particular,
$\kappa(f)\ge2$ in odd dimension, and $\kappa(f)\ge4$ in even
dimension unless $f$ is bent.
For a plateaued nonzero component of $F$, we also have
$\kappa(f_b)\le N/4$.
\end{lemma}

\begin{proof}
Applying Cauchy-Schwarz to the vectors $(W_f(a)^2)_a$ and $(1)_a$ and applying Parseval's identity give
\[
\sum_aW_f(a)^4\ge\frac1N\left(\sum_aW_f(a)^2\right)^2=N^3.
\]
Equality holds exactly when all $W_f(a)^2$ equal $N$, which is
the bent condition. For a plateaued function,
$W_f(a)^4=\lambda^2W_f(a)^2$ holds for every $a$, so
$M_4(f)=\lambda^2N^2$. Since the magnitude $\lambda$ is a
power of two and $\lambda^2\ge N$, we have $\lambda^2\ge2N$ when
$n$ is odd, and $\lambda^2\ge4N$ when $n$ is even and $f$ is not
bent. Finally, a nonzero component of $F$ is nonaffine
by Lemma~\ref{lem:bent}; hence $\lambda\le N/2$, yielding
$\kappa(f_b)\le N/4$.
\end{proof}

\subsection{Odd dimension}

Suppose $n$ is odd. Let $\mathcal Q$ index the nonplateaued components
and let $\mathcal T$ index the plateaued nonzero components, so
$|\mathcal Q|=q$ and $|\mathcal T|=N-1-q$.
Equation~\eqref{eq:kappa-total} gives
\begin{equation}\label{eq:odd-moment}
\sum_{b\in\mathcal Q}\bigl(\kappa(f_b)-1\bigr)
+\sum_{b\in\mathcal T}\bigl(\kappa(f_b)-2\bigr)=q.
\end{equation}
All terms $\kappa(f_b)-1$, resp., $\kappa(f_b)-2$ are nonnegative by Lemma~\ref{lem:kappa}.

\begin{proof}[Proof of Theorem~\ref{thm:count}, odd dimension]
Assume $q>0$. If $b\in\mathcal Q$, then~\eqref{eq:odd-moment}
gives $\kappa(f_b)-1\le q$.
Fix a Walsh coefficient $w=W_{f_b}(a_0)$. Parseval gives
$\sum_{a\ne a_0}W_{f_b}(a)^2=N^2-w^2$, and Cauchy-Schwarz yields
\[
M_4(f_b)\ge w^4+\frac{(N^2-w^2)^2}{N-1}.
\]
By rearranging terms,
\begin{equation*}
\left(\frac{w^2}{N}-1\right)^2
\le (N-1)\bigl(\kappa(f_b)-1\bigr)\le q(N-1),
\end{equation*}
and hence $w^2\le N(1+\sqrt{q(N-1)})$.

Now consider $b\in\mathcal T$. The nonzero Walsh
coefficients of $f_b$ have square $\lambda^2 = N \cdot \kappa(f_b)$. By Lemma~\ref{lem:kappa} and
\eqref{eq:odd-moment}, we have $2\le \kappa(f_b)\le N/4$ and $\kappa(f_b)-2\le q$.
Therefore
\[
\left(\frac{\lambda^2}{N}-1\right)^2 = (\kappa(f_b)-1)^2=\kappa(f_b)(\kappa(f_b)-2)+1\le\frac{Nq}{4}+1\le q(N-1),
\]
where the last inequality uses $N\ge3$ and $q\ge1$.
Thus the same bound holds for plateaued components. Taking the
maximum over all components and all $a$ proves~\eqref{eq:odd-count}.
\end{proof}

When $q=0$, equation~\eqref{eq:odd-moment} forces
$\kappa(f_b)=2$ for every $b\ne0$, recovering the classical
almost-bent conclusion~\cite[Corollary~2]{BergerEtal06}.
For $q>0$, dividing~\eqref{eq:odd-count} by $N^2$ gives
\[
\left(\frac{\Lin(F)}N\right)^2 - \frac1N
\le \sqrt{\frac qN\left(1-\frac1N\right)}.
\]
Since $q\leq N$, this proves the asymptotic statement
following Theorem~\ref{thm:count}.

\subsection{Even dimension}
For this case, choose $b_0\ne0$ and an affine function $\ell$ with
$d_H(f_{b_0},\ell)=t=\NL(F)$, and consider the function graph of $F$ over the subset 
$E=\{x:f_{b_0}(x)\ne\ell(x)\}$.
For $t<N/4$, pairing component indices as $\{b,b+b_0\}$ gives
strong lower bounds on the fourth moments of the restricted character sums whenever
both components are plateaued. Comparing these contributions with
the total contribution fixed by the Sidon identity~\eqref{eq:sidon-fourth}
limits the number of such pairs and yields a lower bound on $q$.

\begin{proposition}\label{prop:even-count}
Suppose $n$ is even and $\Lin(F)>N/2$. Let
$t=(N-\Lin(F))/2$. Then
\begin{equation}\label{eq:even-count}
q\ge\frac N2+1-t-\frac{t^2}{N}.
\end{equation}
\end{proposition}

\begin{proof}
Choose $b_0\ne0$ and an affine function $\ell(x)=a_0\cdot x+c$
with $d_H(f_{b_0},\ell)=t$, and write $g=f_{b_0}+\ell=\ind_E$.
Here $t>0$ by Lemma~\ref{lem:bent}, and $t<N/4$ by the hypothesis of the statement.
The component $f_{b_0}$ is nonplateaued by
Lemma~\ref{lem:plateaued-magnitude}.
For $a,b\in \F_2^n$, let
\[
T_b(a)=\sum_{x\in E}(-1)^{f_b(x)+a\cdot x}
      =\frac{W_{f_b}(a)-W_{f_b+g}(a)}2,
\qquad S_b=\sum_aT_b(a)^4.
\]
The function $u_b(x)=\ind_E(x)(-1)^{f_b(x)}$ has Fourier transform
$T_b$ and satisfies $u_b(x)^2=\ind_E(x)$. Parseval's identity gives
\[
\sum_aT_b(a)^2=N\sum_xu_b(x)^2=N|E|=Nt.
\]
Cauchy-Schwarz therefore gives, for every $b$,
\begin{equation}\label{eq:restricted-baseline}
S_b\ge\frac1N\left(\sum_aT_b(a)^2\right)^2=Nt^2.
\end{equation}
The set $A=\{(x,F(x)):x\in E\}\subseteq\Gamma_F$ is Sidon because
$\Gamma_F$ is Sidon. The character sum $\alpha_{\chi_{a,b}}$ for $(a,b) \in G = \F_2^n \times \F_2^n$ is $T_b(a)$,
so~\eqref{eq:sidon-fourth} yields
\begin{equation}\label{eq:restricted-fourth}
\sum_bS_b=N^2(3t^2-2t).
\end{equation}

Partition $\F_2^n$ into the $N/2$ pairs $\{b,b+b_0\}$.
Since $f_{b_0}=\ell+1$ on $E$, we have, for every $b$,
\[
T_{b+b_0}(a)=-(-1)^cT_b(a+a_0),\qquad S_{b+b_0}=S_b.
\]
Let $p$ count the pairs $\{b,b+b_0\}$ for which $f_b$ and $f_{b+b_0}$ are both plateaued and nonzero.
Every other pair contains an index of a nonplateaued component, including the pair
$\{0,b_0\}$, since $f_{b_0}$ is nonplateaued. The pairs are disjoint, so
\begin{equation}\label{eq:paired-count}
p\ge\frac N2-q.
\end{equation}
For a pair $\{b,b+b_0\}$ counted by $p$, both $f_b$ and
$f_b+g=f_{b+b_0}+\ell$ are plateaued,
since affine addition preserves plateauedness. Moreover, both $f_b$ and $f_b+g$, and therefore both $f_b$ and $f_{b+b_0}$, must have the same
nonzero Walsh magnitude. Otherwise, write the magnitudes as
$\lambda<\lambda'$, so $\lambda'\ge2\lambda$.
The function of magnitude $\lambda'$ has $N^2/(\lambda')^2$
nonzero Walsh coefficients, and Parseval's identity gives
\[
N(N-2t)=N \sum_x (-1)^{f_b(x)}(-1)^{f_b(x)+g(x)} = \sum_aW_{f_b}(a)W_{f_b+g}(a)
\le N^2\frac{\lambda}{\lambda'}\le\frac{N^2}{2},
\]
contrary to $t<N/4$.

Write $\lambda$ for the common magnitude of $f_b$ and $f_{b+b_0}$. Every nonzero $T_b(a)$
has magnitude $\lambda/2$ or $\lambda$. Such a comparison of the
Walsh spectra of two plateaued functions also occurs in Potapov's
distance argument~\cite[proof of Theorem~2(2)]{Potapov20}.
If $\lambda=\sqrt N$, both functions are bent and only the
magnitude $\lambda$ occurs. Otherwise $\lambda\ge2\sqrt N$, since $\lambda$ is a power of two. Consequently, every nonzero $T_b(a)$ has magnitude at least $\sqrt{N}$, yielding the pointwise inequality $T_b(a)^4 \geq N T_b(a)^2$ for every $a$. For every pair counted by $p$, we therefore have
\begin{equation}\label{eq:paired-fourth}
S_b\ge N\sum_aT_b(a)^2=N^2t.
\end{equation}

The pair $\{0,b_0\}$ is not counted by $p$. Since $T_0(0)=t$, we have
\begin{equation}\label{eq:exceptional-fourth}
S_0\ge t^4.
\end{equation}
Sum one value $S_b$ from each pair, using~\eqref{eq:paired-fourth}
for the $p$ plateaued pairs, \eqref{eq:exceptional-fourth} for
$\{0,b_0\}$, and~\eqref{eq:restricted-baseline} for the remaining
$N/2-p-1$ pairs. Equation~\eqref{eq:restricted-fourth} gives
\[
\frac{N^2}{2}(3t^2-2t)
\ge pN^2t+\left(\frac N2-p-1\right)Nt^2+t^4.
\]
Rearranging yields
\begin{align*}
pNt(N-t)&\le N^2t(t-1)+Nt^2-t^4\\
        &=Nt(N-t)\left(t-1+\frac{t^2}{N}\right).
\end{align*}
Dividing by $Nt(N-t)>0$ and using~\eqref{eq:paired-count}
proves~\eqref{eq:even-count}.
\end{proof}

\begin{proof}[Proof of Theorem~\ref{thm:count}, even dimension]
Fix $1/2\le C<1$ and suppose that $\Lin(F)>CN$. Then
\[
0<t=\frac{N-\Lin(F)}2<\frac{(1-C)N}{2}.
\]
Since $\Lin(F)>N/2$, Proposition~\ref{prop:even-count} gives
\begin{align*}
q&\ge\frac N2+1-t-\frac{t^2}{N}\\
 &>\frac N2+1-\frac{(1-C)N}{2}-\frac{(1-C)^2N}{4}\\
 &=\frac{4C-C^2-1}{4}\,N+1.
\end{align*}
This proves the implication.
\end{proof}

For the stated special case in~(\ref{eq:proportion-special}), if $\Lin(F)>N/2$, then $t$ is an
integer with $t\le N/4-1$. The expression $t+t^2/N$ is increasing
for $t\ge1$. Proposition~\ref{prop:even-count} therefore gives
\begin{equation*}
q\ge\frac N2+1-\left(\frac N4-1\right)
                 -\frac{(N/4-1)^2}{N}
 =\frac{3N}{16}+\frac52-\frac1N
 >\frac{3N}{16}+2,
\end{equation*}
where the last inequality uses $N\ge3$.

\begin{remark}
An APN function cannot have exactly one nonplateaued component.
For a Boolean function $f$, write
\begin{equation*}
Q_f(x)=\sum_{u,v\in \F_2^n}(-1)^{f(x)+f(x+u)+f(x+v)+f(x+u+v)}.
\end{equation*}
The classical plateaued characterization says that $Q_f$ is constant
if and only if $f$ is plateaued
\cite[Theorem~1]{CarletProuff03}. For an APN function $F$, character
orthogonality gives
\begin{equation}\label{eq:plateaued}
\sum_{b\ne0}Q_{f_b}(x)=N(3N-2)-N^2=2N(N-1).
\end{equation}
For each fixed $x$, the second derivatives $D_uD_vf_b(x)$
vanish simultaneously for all $b\ne0$ precisely when $u=0$,
$v=0$, or $u=v$, giving $3N-2$ ordered pairs $(u,v)$.
Indeed, for $u\ne0$, simultaneous vanishing is equivalent to
$D_uF(x)=D_uF(x+v)$, and the APN property forces
$v\in\{0,u\}$.

If $f_{b_0}$ were the only nonplateaued nonzero component,
then $Q_{f_b}$ would be constant for every $b \notin \{0,b_0\}$.
Equation~(\ref{eq:plateaued}) would therefore force
\[
Q_{f_{b_0}}(x)
=2N(N-1)-\sum_{\substack{b\ne0\\b\ne b_0}}Q_{f_b}(x)
\]
to be independent of $x$ as well. By the plateaued
characterization, $f_{b_0}$ would then also be plateaued,
a contradiction.
\end{remark}

\section*{Use of Generative AI and Contributions of the Authors}
ChatGPT Pro (GPT-6 Astra and Codex) was used by both authors in the research and preparation of this article. Separate AI-based investigations of three open problems concerning APN functions were combined, and the most promising results were selected for further development. AI contributions included proposing and strengthening mathematical statements, developing and simplifying proofs, searching the literature and examining attribution, and drafting and iteratively revising a preliminary version of the manuscript. 

The first named author directed the initial investigation, selected the results, their formulation, and scope of the article, and published a preprint containing an initial version of the results in this manuscript. Using the same AI, both the first and second named authors were then able to improve several results in this manuscript, resulting in this joint manuscript.

Both authors carefully reviewed and revised the final manuscript to ensure correctness of results, proper attribution, and a clear exposition. The authors take full responsibility for its content.


\begin{thebibliography}{99}

\bibitem{BabaiSos85}
L.~Babai and V.~T.~S\'os,
Sidon sets in groups and induced subgraphs of Cayley graphs,
\emph{Eur. J. Comb.} \textbf{6} (1985), no.~2, 101-114.
\href{https://doi.org/10.1016/S0195-6698(85)80001-9}{doi:10.1016/S0195-6698(85)80001-9}.

\bibitem{BeierleLeander22}
C.~Beierle and G.~Leander,
New instances of quadratic APN functions,
\emph{IEEE Trans. Inf. Theory} \textbf{68} (2022), no.~1,
670-678.
\href{https://doi.org/10.1109/TIT.2021.3120698}{doi:10.1109/TIT.2021.3120698}.

\bibitem{BeneteauEtal26}
S.~H.~B\'en\'eteau, N.~Goluboff, L.~K\"olsch and D.~Vaghasiya,
On the Walsh spectra of quadratic APN functions,
\emph{IEEE Trans. Inf. Theory} \textbf{72} (2026), no.~7,
5207--5216.
\href{https://doi.org/10.1109/TIT.2026.3695003}{doi:10.1109/TIT.2026.3695003}.

\bibitem{BergerEtal06}
T.~P.~Berger, A.~Canteaut, P.~Charpin and Y.~Laigle-Chapuy,
On almost perfect nonlinear functions over $\F_2^n$,
\emph{IEEE Trans. Inf. Theory} \textbf{52} (2006), no.~9,
4160-4170.
\href{https://doi.org/10.1109/TIT.2006.880036}{doi:10.1109/TIT.2006.880036}.

\bibitem{BrinkmannLeander08}
M.~Brinkmann and G.~Leander,
On the classification of APN functions up to dimension five,
\emph{Des. Codes Cryptogr.} \textbf{49} (2008), 273-288.
\href{https://doi.org/10.1007/s10623-008-9194-6}{doi:10.1007/s10623-008-9194-6}.

\bibitem{BrouwerTolhuizen93}
A.~E.~Brouwer and L.~M.~G.~M.~Tolhuizen,
A sharpening of the Johnson bound for binary linear codes and the
nonexistence of linear codes with Preparata parameters,
\emph{Des. Codes Cryptogr.} \textbf{3} (1993), 95-98.
\href{https://doi.org/10.1007/BF01388407}{doi:10.1007/BF01388407}.

\bibitem{BudaghyanEtal18}
L.~Budaghyan, C.~Carlet, T.~Helleseth, N.~Li and B.~Sun,
On upper bounds for algebraic degrees of APN functions,
\emph{IEEE Trans. Inf. Theory} \textbf{64} (2018), no.~6,
4399-4411.
\href{https://doi.org/10.1109/TIT.2017.2757938}{doi:10.1109/TIT.2017.2757938}.

\bibitem{Canteaut01}
A.~Canteaut,
Cryptographic functions and design criteria for block ciphers,
in \emph{Progress in Cryptology---INDOCRYPT 2001},
C.~P. Rangan and C.~Ding (eds.),
Lecture Notes in Computer Science, vol.~2247,
Springer, Berlin, Heidelberg, 2001, pp.~1-16.
\href{https://doi.org/10.1007/3-540-45311-3_1}
{doi:10.1007/3-540-45311-3\_1}.

\bibitem{CCD99}
A.~Canteaut, P.~Charpin, and H.~Dobbertin,
A new characterization of almost bent functions,
in \emph{Fast Software Encryption}, L.~Knudsen (ed.),
Lecture Notes in Computer Science, vol.~1636,
Springer, Berlin, Heidelberg, 1999, pp.~186-200.
\href{https://doi.org/10.1007/3-540-48519-8_14}
     {doi:10.1007/3-540-48519-8\_14}.

\bibitem{Carlet10}
C.~Carlet,
Vectorial Boolean functions for cryptography,
in \emph{Boolean Models and Methods in Mathematics, Computer Science,
and Engineering}, Y.~Crama and P.~L.~Hammer (eds.),
Cambridge University Press, Cambridge, 2010, pp.~398--470.
\href{https://doi.org/10.1017/CBO9780511780448.012}{doi:10.1017/CBO9780511780448.012}.

\bibitem{Carlet15}
C.~Carlet,
Boolean and vectorial plateaued functions and APN functions,
\emph{IEEE Trans. Inf. Theory} \textbf{61} (2015), no.~11,
6272-6289.
\href{https://doi.org/10.1109/TIT.2015.2481384}{doi:10.1109/TIT.2015.2481384}.

\bibitem{Carlet18}
C.~Carlet,
Characterizations of the differential uniformity of vectorial
functions by the Walsh transform,
\emph{IEEE Trans. Inf. Theory} \textbf{64} (2018), no.~9,
6443--6453.
\href{https://doi.org/10.1109/TIT.2017.2761392}
{doi:10.1109/TIT.2017.2761392}.

\bibitem{CarletBook}
C.~Carlet,
\emph{Boolean Functions for Cryptography and Coding Theory},
Cambridge University Press, Cambridge, 2021.
\href{https://doi.org/10.1017/9781108606806}{doi:10.1017/9781108606806}.

\bibitem{Carlet21}
C.~Carlet,
On the properties of the Boolean functions associated to the
differential spectrum of general APN functions and their consequences,
\emph{IEEE Trans. Inf. Theory} \textbf{67} (2021), no.~10,
6926-6939.
\href{https://doi.org/10.1109/TIT.2021.3081139}{doi:10.1109/TIT.2021.3081139}.

\bibitem{CCZ98}
C.~Carlet, P.~Charpin and V.~Zinoviev,
Codes, bent functions and permutations suitable for DES-like cryptosystems,
\emph{Des. Codes Cryptogr.} \textbf{15} (1998), 125-156.
\href{https://doi.org/10.1023/A:1008344232130}{doi:10.1023/A:1008344232130}.

\bibitem{CarletPicek23}
C.~Carlet and S.~Picek,
On the exponents of APN power functions and Sidon sets, sum-free sets,
and Dickson polynomials,
\emph{Adv. Math. Commun.} \textbf{17} (2023), no.~6, 1507-1525.
\href{https://doi.org/10.3934/amc.2021064}{doi:10.3934/amc.2021064}.

\bibitem{CarletProuff03}
C.~Carlet and E.~Prouff,
On plateaued functions and their constructions,
in \emph{Fast Software Encryption---FSE 2003},
T.~Johansson (ed.),
Lecture Notes in Computer Science, vol.~2887,
Springer, Berlin, Heidelberg, 2003, pp.~54--73.
\href{https://doi.org/10.1007/978-3-540-39887-5_6}
{doi:10.1007/978-3-540-39887-5\_6}.

\bibitem{CarletThornburgh26}
C.~Carlet and D.~Thornburgh,
Resolving a conjecture on quadratic APN functions and a new quadratic
$(n,n)$-function associated to crooked functions,
\href{https://arxiv.org/abs/2608.23888v2}{arXiv:2608.23888v2},
revised 7 September 2026; first posted 24 August 2026.

\bibitem{CV94}
F.~Chabaud and S.~Vaudenay,
Links between differential and linear cryptanalysis,
in \emph{Advances in Cryptology---EUROCRYPT '94},
A.~De Santis (ed.),
Lecture Notes in Computer Science, vol.~950,
Springer, Berlin, Heidelberg, 1995, pp.~356--365.
\href{https://doi.org/10.1007/BFb0053450}
{doi:10.1007/BFb0053450}.

\bibitem{CP26}
I.~Czerwinski and A.~Pott,
On large Sidon sets,
\emph{J. Comb. Theory, Ser. A} \textbf{220} (2026), article~106129.
\href{https://doi.org/10.1016/j.jcta.2025.106129}{doi:10.1016/j.jcta.2025.106129}.

\bibitem{GillotLangevin26}
V.~Gillot and P.~Langevin,
On known APNs,
\href{https://arxiv.org/abs/2601.11247}{arXiv:2601.11247}, 2026.

\bibitem{GillotLangevinLo25}
V.~Gillot, P.~Langevin and A.~Lo,
Spectral moment of order four and the uniqueness of the CCZ class of Dublin APN permutation,
\href{https://arxiv.org/abs/2507.12853}{arXiv:2507.12853}, 2025.

\bibitem{KolschPolujan24}
L.~K\"olsch and A.~Polujan,
Value distributions of perfect nonlinear functions,
\emph{Combinatorica} \textbf{44} (2024), 231--268.
\href{https://doi.org/10.1007/s00493-023-00067-y}
{doi:10.1007/s00493-023-00067-y}.

\bibitem{KolschPolujan26}
L.~K\"olsch and A.~Polujan,
The combinatorial structure and value distributions of plateaued functions,
\emph{J. Cryptol.} \textbf{39} (2026), article~21.
\href{https://doi.org/10.1007/s00145-026-09578-5}{doi:10.1007/s00145-026-09578-5}.

\bibitem{Nyberg91}
K.~Nyberg,
Perfect nonlinear S-boxes,
in \emph{Advances in Cryptology---EUROCRYPT '91},
D.~W.~Davies (ed.),
Lecture Notes in Computer Science, vol.~547,
Springer, Berlin, Heidelberg, 1991, pp.~378--386.
\href{https://doi.org/10.1007/3-540-46416-6_32}
{doi:10.1007/3-540-46416-6\_32}.

\bibitem{NybergKnudsen93}
\enlargethispage{2\baselineskip}
K.~Nyberg and L.~R.~Knudsen,
Provable security against differential cryptanalysis,
in \emph{Advances in Cryptology---CRYPTO '92},
E.~F.~Brickell (ed.),
Lecture Notes in Computer Science, vol.~740,
Springer, Berlin, Heidelberg, 1993, pp.~566--574.
\href{https://doi.org/10.1007/3-540-48071-4_41}
{doi:10.1007/3-540-48071-4\_41}.

\bibitem{Potapov20}
V.~N.~Potapov,
On $q$-ary bent and plateaued functions,
\emph{Des. Codes Cryptogr.} \textbf{88} (2020), 2037-2049.
\href{https://doi.org/10.1007/s10623-020-00761-8}{doi:10.1007/s10623-020-00761-8}.

\bibitem{Rothaus76}
O.~S.~Rothaus,
On ``bent'' functions,
\emph{J. Comb. Theory, Ser. A} \textbf{20} (1976), no.~3, 300-305.
\href{https://doi.org/10.1016/0097-3165(76)90024-8}{doi:10.1016/0097-3165(76)90024-8}.

\bibitem{TaoVu06}
T.~Tao and V.~H.~Vu,
\emph{Additive Combinatorics},
Cambridge University Press, Cambridge, 2006.
\href{https://doi.org/10.1017/CBO9780511755149}{doi:10.1017/CBO9780511755149}.

\bibitem{Thornburgh26}
D.~Thornburgh,
A new upper bound for Sidon sets in $\F_2^{4k+3}$,
\href{https://arxiv.org/abs/2609.27731}{arXiv:2609.27731}, 2026.

\bibitem{ZhengZhang99}
Y.~Zheng and X.-M.~Zhang,
On plateaued functions,
\emph{IEEE Trans. Inf. Theory} \textbf{47} (2001), no.~3,
1215--1223.
\href{https://doi.org/10.1109/18.915690}
{doi:10.1109/18.915690}.

\end{thebibliography}
\end{document}